\documentclass[12pt]{amsart}
\usepackage{SSdefn}

\newcommand{\fgen}{{\rm fg}}

\newcommand{\VIC}{\mathbf{VIC}}

\newcommand{\FI}{\mathbf{FI}}
\newcommand{\HF}{\mathrm{HF}}

\author{Steven V Sam}
\address{Department of Mathematics, University of Wisconsin, Madison, WI}
\email{\href{mailto:svs@math.wisc.edu}{svs@math.wisc.edu}}
\urladdr{\url{http://math.wisc.edu/~svs/}}
\thanks{SS was supported by NSF grant DMS-1500069.}

\author{Andrew Snowden}
\address{Department of Mathematics, University of Michigan, Ann Arbor, MI}
\email{\href{mailto:asnowden@umich.edu}{asnowden@umich.edu}}
\urladdr{\url{http://www-personal.umich.edu/~asnowden/}}
\thanks{AS was supported by NSF grants DMS-1303082 and DMS-1453893 and a Sloan Fellowship.}

\title{Some generalizations of Schur functors}


\begin{document}

\begin{abstract}
The theory of Schur functors provides a powerful and elegant approach to the representation theory of $\GL_n$---at least to the so-called polynomial representations---especially to questions about how the theory varies with $n$. We develop parallel theories that apply to other classical groups and to non-polynomial representations of $\GL_n$. These theories can also be viewed as linear analogs of the theory of $\FI$-modules.
\end{abstract}

\maketitle
%\tableofcontents

\section{Introduction}

The theory of Schur functors is a major achievement of representation theory. One can use these functors to construct the irreducible representations of $\GL_n$---at least those whose highest weight is positive (the so-called ``polynomial'' representations). This point of view is important because it gives a clear picture of how the irreducible representation of $\GL_n$ with highest weight $\lambda$ varies with $n$, which is crucial to understand when studying problems in which $n$ varies. Unfortunately, this theory only applies to $\GL_n$ and not other classical groups, and, as mentioned, only to the polynomial representations of $\GL_n$. The purpose of this paper is to develop parallel theories that apply in other cases.

To explain our results, we focus on the orthogonal case. Let $\cC=\cC_{\bO}$ be the category whose objects are finite dimensional complex vector spaces equipped with non-degenerate symmetric bilinear forms and whose morphisms are linear isometries. An {\bf orthogonal Schur functor} is a functor from $\cC$ to the category of vector spaces that is algebraic, in an appropriate sense. We denote the category of such functors by $\cA=\cA_{\bO}$. The main results of this paper elucidate the structure of this category, and its analogs in other cases.

Let $\cS$ be the category of (classical) Schur functors. Given an object of  $\cS$, one can evaluate on $\bC^{\infty}$ and obtain a polynomial representation of $\GL_{\infty}$. The resulting functor $\cS \to \Rep^{\pol}(\GL)$ is an equivalence of categories. This is a useful perspective since representations of a group are more tangible than functors: it is not difficult to show that $\Rep^{\pol}(\GL)$ is semi-simple and that the simple objects are indexed by partitions; from this, one deduces the structure of $\cS$.

We take a similar approach to study $\cA$. Given an object of $\cA$, one can again evaluate on $\bC^{\infty}$ (equipped with its standard symmetric bilinear form) and obtain a representation of $\bO_{\infty}$. (Technically, we cannot evaluate an object of $\cA$ on $\bC^{\infty}$, but we can evaluate on $\bC^n$ and take the direct limit.) This representation is algebraic in the sense of \cite{infrank}. We thus have a functor
\begin{displaymath}
T \colon \cA \to \Rep(\bO).
\end{displaymath}
There are two important differences from the previous paragraph that complicate our task. First, the category $\Rep(\bO)$ is not semi-simple. And second, the functor $T$ is not an equivalence. To see this, note that all maps in $\cC$ are injections. Given a representation $V$ of $\bO_n$, we can therefore build a functor $F \colon \cC \to \Vec$ such that $F(\bC^n)=V$ and $F(\bC^k)=0$ for $k \ne n$. This $F$ is a non-zero object of $\cA$ but $T(F)=0$.

The first issue was addressed by \cite{infrank}, which describes the algebraic representation theory of $\bO_{\infty}$ quite clearly. (See \cite{koszulcategory,penkovserganova,penkovstyrkas} for similar work.) The real work in this paper goes in to addressing the second issue.

We now summarize our main results. We first show that $T$ induces an equivalence between the Serre quotient $\cA/\cA_{\tors}$ and $\Rep(\bO)$. Here $\cA_{\tors}$ denotes the category of torsion objects in $\cA$; these are (essentially) those functors $F$ for which $F(\bC^n)=0$ for $n \gg 0$, such as the functor $F$ constructed above. We thus see that $\cA$ is more or less built out of two pieces: $\Rep(\bO)$, which we understand from \cite{infrank}, and $\cA_{\tors}$, which is not difficult to understand directly.

We next prove the somewhat technical result that the section functor $S$ (the right adjoint to $T$) and its derived functors have amenable finiteness properties. Using this, we deduce that $\cA$ is locally noetherian. This is the most important finiteness property of $\cA$.

We then go on to prove a few more results on the structure of $\cA$: we classify the indecomposable injectives, compute the Grothendieck group, define a theory of local cohomology, and introduce a Hilbert series and prove a rationality theorem for it.

All of the above work takes place in \S \ref{s:orth}. In \S \ref{s:other}, we explain how the theory works in other situations (e.g., for symplectic groups). We also see in \S \ref{s:other} that we can interpret orthogonal Schur functors as a sort of linear analog of the $\FI$-modules of Church--Ellenberg--Farb \cite{fimodule}. In fact, many of the results in this paper are analogous to results on $\FI$-modules obtained in \cite{fimodule} and \cite{symc1}.

\section{Orthogonal Schur functors} \label{s:orth}

\subsection{First definitions and results}

Let $\cC=\cC_{\bO}$ be the following category. An object is a pair $(V,\omega)$ consisting of a finite dimensional complex vector space $V$ and a non-degenerate symmetric bilinear form $\omega$ on $V$. A morphism $(V,\omega) \to (V',\omega')$ is a linear map $f \colon V \to V'$ such that $\omega'(f(v_1),f(v_2))=\omega(v_1,v_2)$ for all $v_1,v_2 \in V$. This category appears in \cite[(4.4.7)]{infrank} as $T_1$. A {\bf $\cC$-module} is a functor $\cC \to \Vec$. We denote the category of $\cC$-modules by $\Mod_{\cC}$. For $n \ge 0$, we let $T^n$ be the $\cC$-module given by $T^n(V)=V^{\otimes n}$. For a partition $\lambda$, we let $\bS_{\lambda}$ be the $\cC$-module defined by the Schur functor $\bS_{\lambda}$; note that $T^n$ decomposes as a direct sum of $\bS_{\lambda}$'s. A $\cC$-module is {\bf algebraic} if it is a subquotient of a (possibly infinite) direct sum of $T^n$'s. We denote the category of algebraic $\cC$-modules by $\cA=\cA_{\bO}$. We call this the {\bf category of orthogonal Schur functors}. Both $\cA$ and $\Mod_{\cC}$ are Grothendieck abelian categories.

Suppose that $M$ is a $\cC$-module. Then $M_n=M(\bC^n)$ is a representation of the orthogonal group $\bO_n(\bC)$, regarded as a discrete group. (Here we let $\bC^n$ denote the object $(\bC^n,\omega)$ of $\cC$, where $\omega$ is the standard form on $\bC^n$ for which the standard basis vectors $e_i$ are orthonormal.) The standard inclusion $\bC^n \to \bC^{n+1}$ defines a morphism in $\cC$, and thus a linear map $M_n \to M_{n+1}$. This map is clearly $\bO_n(\bC)$-equivariant. Furthermore, the image of the induced map $M_n \to M_{n+m}$ is contained in the space of $H_{n+m,n}$-invariants, where $H_{n+m,n} \subset \bO_{n+m}$ is the subgroup of $g$ such that $gi = i$ where $i \colon \bC^n \to \bC^{n+m}$ is the standard inclusion.

Let $\cB$ be the category of sequences as above: an object of $\cB$ is a pair $(\{M_n\}_{n \ge 0}, \{i_n\}_{n \ge 0})$ where $M_n$ is a representation of $\bO_n(\bC)$ and $i_n \colon M_n \to M_{n+1}$ is an $\bO_n(\bC)$-equivariant linear map such that the image of the induced map $M_n \to M_{n+m}$ is contained in the $H_{n+m,n}$-invariants. Morphisms in $\cB$ are the obvious things. The previous paragraph shows that $M \mapsto \{M(\bC^n)\}$ defines a functor $\Mod_{\cC} \to \cB$.

\begin{proposition} \label{prop:seqeq}
The functor $\Mod_{\cC} \to \cB$ is an equivalence.
\end{proposition}

\begin{proof}
Without loss of generality, we may replace $\cC$ with its skeletal subcategory on the objects $\bC^n$ for $n \ge 0$. Suppose we are given a sequence $M=(\{M_n\}_{n \ge 0}, \{i_n\}_{n \ge 0})$ that is an object of $\cB$. We define an associated $\cC$-module $F$. On objects, we put $F(\bC^n) = M_n$. Given a morphism $f \colon \bC^n \to \bC^{n+m}$, we can find $g \in \bO_{n+m}$ such that $gf$ is the standard inclusion, and we set $F(f) = g^{-1} i_{n+m-1} \cdots i_{n+1} i_n$. Furthermore, given any other $g' \in \bO_{n+m}$ with the same property, $g^{-1}g' \in H_{n+m,n}$, so $F(f)$ is independent of the particular choice of $g$. The construction $M \mapsto F$ defines a functor $\cB \to \Mod_\cC$ that is quasi-inverse to $\Mod_{\cC} \to \cB$.
\end{proof}

There is a notion of finite generation for $\cC$-modules. Let $M$ be a $\cC$-module, and suppose that we have elements $\{x_i\}_{i \in I}$ where $x_i \in M(V_i)$. The submodule {\bf generated} by the $x_i$ is the smallest $\cC$-submodule of $M$ containing the $x_i$. We say that $M$ is {\bf finitely generated} if it can be generated by a finite set of elements. We note that if $M$ is finitely generated and algebraic then $M(\bC^n)$ is a finite dimensional algebraic representation of $\bO_n$ for all $n$.

Let $M$ be a $\cC$-module. An element $x \in M(\bC^n)$ is {\bf torsion} if $f_*(x)=0$ for some morphism $f \colon \bC^n \to \bC^m$ in $\cC$. The module $M$ is {\bf torsion} if all of its elements are. We let $\cA_{\tors}$ be the category of torsion algebraic $\cC$-modules. We note that a finitely generated torsion object of $\cA_{\tors}$ has finite length.

\subsection{Algebraic representations of $\bO_{\infty}$} \label{ss:repO}

Let $\bO_{\infty}(\bC)=\bigcup_{n \ge 1} \bO_n(\bC)$. Let $\Rep(\bO_\infty(\bC))$ be the category of all representations of the discrete group $\bO_\infty(\bC)$. Let $T^n_{\infty}=(\bC^{\infty})^{\otimes n}$, which is naturally a representation of $\bO_{\infty}(\bC)$. We say that a representation of $\bO_{\infty}(\bC)$ is {\bf algebraic} if it appears as a subquotient of a direct sum of $T^n_{\infty}$'s. We write $\Rep(\bO)$ for the category of algebraic representations of $\bO_{\infty}$, and $\Rep(\bO_n)$ for the category of algebraic representations of $\bO_n$. The following theorem summarizes some of the main results from \cite{infrank} (see \cite{koszulcategory, penkovserganova, penkovstyrkas} for similar results):

\begin{theorem} \label{thm:repO}
The following statements hold in $\Rep(\bO)$:
\begin{enumerate}[\rm (a)]
\item The objects $T^n_{\infty}$ have finite length.
\item The objects $T^n_{\infty}$ are injective, and every object of $\Rep(\bO)$ embeds into a direct sum of $T^n_{\infty}$'s.
\item The indecomposable injectives are exactly the Schur functors $\bS_{\lambda}(\bC^{\infty})$.
\item Let $L^{\lambda}_{\infty}$ be the socle of $\bS_{\lambda}(\bC^{\infty})$. Then the $L^{\lambda}_{\infty}$'s are exactly the simple objects.
\item Every finite length object of $\Rep(\bO)$ has finite injective dimension.
\end{enumerate}
\end{theorem}

\begin{proof}
\begin{enumerate}[label=(\alph*),leftmargin=*]
\item \cite[Proposition 4.1.5]{infrank}
\end{enumerate}
\begin{enumerate}[label=(\alph*),resume,topsep=0pt]
\item Follows from (c) since $T^n_\infty$ is a finite direct sum of Schur functors.
\item \cite[Proposition 4.2.9]{infrank}
\item \cite[Proposition 4.1.4]{infrank}
\item Combine \cite[(2.1.5)]{infrank} and \cite[Theorem 4.2.6]{infrank}. \qedhere
\end{enumerate}
\end{proof}

Let $V$ be a representation of $\bO_{\infty}(\bC)$. Let $G_n$ and $H_n$ be the subgroups of $\bO_{\infty}$ consisting of matrices of the form
\begin{displaymath}
\begin{pmatrix} \ast & 0 \\ 0 & 1 \end{pmatrix}, \qquad
\begin{pmatrix} 1 & 0 \\ 0 & \ast \end{pmatrix},
\end{displaymath}
where the top left block is $n \times n$ and the bottom right block is $(\infty-n) \times (\infty-n)$. Put $\wt{\Gamma}_n(V) = V^{H_n}$. Since $G_n \cong \bO_n(\bC)$ commutes with $H_n$, we see that $\wt{\Gamma}_n(V)$ is a representation of $\bO_n(\bC)$. We let $\Gamma_n$ be the restriction of $\wt{\Gamma}_n$ to $\Rep(\bO)$. It follows from Theorem~\ref{thm:repO}(b) and Theorem~\ref{thm:Gamma}(b) below that $\Gamma_n$ takes values in $\Rep(\bO_n)$. We thus have a left-exact functor
\begin{displaymath}
\Gamma_n \colon \Rep(\bO) \to \Rep(\bO_n),
\end{displaymath}
called the {\bf specialization functor}. The following theorem summarizes some of its main properties:

\begin{theorem} \label{thm:Gamma}
We have the following:
\begin{enumerate}[\rm \indent (a)]
\item $\Gamma_n$ is a left-exact symmetric monoidal functor.
\item $\Gamma_n(T^r_{\infty})=(\bC^n)^{\otimes r}$ and $\Gamma_n(\bS_{\lambda}(\bC^{\infty}))=\bS_{\lambda}(\bC^n)$.
\item $\Gamma_n(L^{\lambda}_{\infty})$ is the irreducible representation of $\bO_n$ with highest weight $\lambda$ if $n \ge \lambda_1^{\dag}+\lambda_2^{\dag}$, and $0$ otherwise. 
\item Suppose $V \in \Rep(\bO)$ has finite length. Then $\rR^i \Gamma_n(V)$ is finite dimensional for all $i$ and $n$, and vanishes for $i \gg 0$. Furthermore, for $n \gg 0$ we have $\rR^i \Gamma_n(V)=0$ for all $i>0$.
\item There is an explicit combinatorial rule to compute $\rR^i \Gamma_n(L^{\lambda}_{\infty})$. In particular, it is nonzero for at most one value of $i$.
\end{enumerate}
\end{theorem}

\begin{proof}
\begin{enumerate}[label=(\alph*),leftmargin=*]
\item \cite[(4.4.4), (4.4.5)]{infrank}
\end{enumerate}
\begin{enumerate}[label=(\alph*),resume,topsep=0pt]
\item For $r=1$ of the first statement, this follows from the construction of $\Gamma_n$ in \cite[(4.4.4)]{infrank}, the rest of the statement follows since $\Gamma_n$ is a symmetric monoidal functor.
\item This is a special case of (e), but see also \cite[\S 19.5]{fultonharris}.
\item For $V$ simple this follows from (e). The general case follows from d\'evissage. \item See \cite[(4.4.6)]{infrank}. \qedhere
\end{enumerate}
\end{proof}

We will also require the following property of derived specialization:

\begin{proposition} \label{prop:colimGamma}
The functor $\rR^i \Gamma_n$ commutes with filtered colimits.
\end{proposition}

\begin{proof}
We first treat the case $n=0$, where $\Gamma_n$ is simply the functor that assigns to $V$ the invariant subspace $V^{\bO_{\infty}}$. We thus see that $\Gamma_0(V)=\Hom_{\Rep(\bO)}(\bC, V)$, where $\bC$ is the trivial representation, and so $\rR^i \Gamma_0(V)=\Ext^i_{\Rep(\bO)}(\bC, V)$. In \cite[(4.3.1)]{infrank}, we show that $\Rep(\bO)$ is equivalent to the category $\Mod_A^0$ of modules over the twisted commutative algebra $A=\Sym(\Sym^2(\bC^{\infty}))$ that are supported at~0 (i.e., locally annihilated by a power of the maximal ideal $A_+$). Under this equivalence, $\bC$ corresponds to the module $\bC=A/A_+$. It thus suffices to show that $\Ext^i_{\Mod_A^0}(\bC, -)$ commutes with filtered colimits.

We claim that injective objects of $\Mod_A^0$ remain injective in $\Mod_A$. Suppose that $I$ is a finite length injective object of $\Mod_A^0$, and consider an injection $I \to M$ in $\Mod_A$. Every $A$-module is canonically graded. Let $M_{\ge n}$ be the submodule $\bigoplus_{i \ge n} M_i$ of $M$, and let $M^{\le n}$ be the quotient $M/M_{\ge n}$. The formation of $M^{\le n}$ is clearly exact in $M$. Moreover, $M^{\le n}$ is supported at~0, being annihilated by the $n$th power of $A_+$. Since $I$ has finite length, we have $I^{\le n}=I$ for $n \gg 0$. Thus, for appropriate $n$, we have an injection $I=I^{\le n} \to M^{\le n}$ in $\Mod_A^0$. Since $I$ is injective in $\Mod_A^0$, this injection splits; composing with the canonical surjection $M \to M^{\le n}$ splits the original injection. Thus $I$ is injective in $\Mod_A$. For the general case, simply observe that all injectives of $\Mod_A^0$ are direct sums of finite length injectives (this follows from Theorem~\ref{thm:repO}(c), for instance), and arbitrary direct sums of injective objects of $\Mod_A$ are injective, since this category is locally noetherian \cite[Theorem 1.1]{sym2noeth}.

By the previous paragraph, we have $\Ext^i_{\Mod_A^0}(-,-)=\Ext^i_{\Mod_A}(-,-)$, and so it suffices to show that $\Ext^i_{\Mod_A}(\bC, -)$ commutes with filtered colimits. Let $P_{\bullet} \to \bC$ be the Koszul resolution of $\bC$ as an $A$-module; thus $P_i=A \otimes \lw^i(\Sym^2(\bC^{\infty}))$. Note that each $P_i$ is finitely generated as an $A$-module. We thus see that $\Ext^{\bullet}_{\Mod_A}(\bC, M)$ is computed by the complex $\Hom_A(P_{\bullet}, M)$. Since each $P_{\bullet}$ is finitely presented (being finitely generated and projective), formation of this complex commutes with filtered colimits. The result now follows from the fact that filtered colimits are exact.

We now treat the case of arbitrary $n$. Ignoring the $\bO_n$-structure, the functor $\Gamma_n$ factors as
\begin{displaymath}
\xymatrix{
\Rep(\bO) \ar[r]^{R_n} & \Rep(\bO) \ar[r]^{\Gamma_0} & \Vec, }
\end{displaymath}
where $R_n$ is the functor $R_n(V)=V \vert_{H_n}$ (note that $H_n \cong \bO_{\infty})$. The functor $R_n$ obviously is exact and commutes with colimits. Furthermore, it takes injective objects to injective objects, as can be seen from the explicit description of injectives in Theorem~\ref{thm:repO}(c). We thus see that $\rR^i \Gamma_n = \rR^i \Gamma_0 \circ R_n$, and so the result follows.
\end{proof}

\subsection{The relation between $\cA$ and $\Rep(\bO)$}

We define functors
\begin{displaymath}
\wt{T} \colon \Mod_{\cC} \to \Rep(\bO_{\infty}(\bC)), \qquad
\wt{S} \colon \Rep(\bO_{\infty}(\bC)) \to \Mod_{\cC}
\end{displaymath}
by
\begin{displaymath}
\wt{T}(M)=\varinjlim_{n \to \infty} M(\bC^n), \qquad
\wt{S}(V)=\{\wt{\Gamma}_n(V)\}_{n \ge 0}.
\end{displaymath}
In the definition of $\wt{S}$, the transition maps are the obvious inclusions $\wt{\Gamma}_n(V) \subset \wt{\Gamma}_{n+1}(V)$. It is clear that the image of $\wt{\Gamma}_n(V)$ in $\wt{\Gamma}_{n+m}(V)$ is contained in the $H_{n+m,n}$-invariants. We note that $\wt{T}$ is exact, as direct limits are exact, and $\wt{S}$ is left-exact, as each $\wt{\Gamma}_n$ is. Moreover, $\wt{S}$ commutes with arbitrary filtered colimits. We let $T$ be the restriction of $\wt{T}$ to $\cA$ and $S$ the restriction of $\wt{S}$ to $\Rep(\bO)$.

\begin{theorem} \label{thm:ST}
We have the following:
\begin{enumerate}[\rm \indent (a)]
\item $T(T^n)=T^n_{\infty}$ and $S(T^n_{\infty})=T^n$, and similarly for $\bS_{\lambda}$.
\item $T$ takes values in $\Rep(\bO)$ and $S$ takes values in $\cA$.
\item $(T,S)$ is an adjoint pair and the unit $\id \to ST$ is an isomorphism.
\item We have $(\rR^i S)(V)=\{\rR^i \Gamma_n(V)\}_{n \ge 0}$ with obvious transition maps.
\item $T$ induces an equivalence of categories $\cA/\cA_{\tors} \to \Rep(\bO)$.
\end{enumerate}
\end{theorem}

\begin{proof}
(a) The statement for $T$ is clear. It is clear that $T^n(\bC^k) \subset (T^n_{\infty})^{H_k}$ for all $k$, and the these inclusions are compatible with transition maps. By Theorem~\ref{thm:Gamma}(b), these inclusions are equalities. Thus $S(T_n^{\infty})=T_n$.

(b) Since $T$ is exact and takes $T_n$ to $T^n_{\infty}$, it takes subquotients of sums of $T^n$'s to subquotients of sums of $T^n_{\infty}$'s. Let $V$ be an algebraic representation of $\Rep(\bO)$. By Theorem~\ref{thm:repO}(b), $V$ embeds into a direct sum of $T^n_{\infty}$'s. Since $S$ is left-exact and commutes with arbitrary direct sums, we see that $V$ embeds into a direct sum of $T^n$'s and is therefore algebraic.

(c) Let $M$ be an object of $\cA$ and $V$ a representation of $\bO_{\infty}(\bC)$. Suppose that $f \colon T(M) \to V$ is a $\bO_{\infty}(\bC)$-equivariant map. There is a canonical $\bO_n(\bC)$-equivariant map $M_n=M(\bC^n) \to T(M)$, and so composing with $f$ induces a map $f_n \colon M_n \to V$ that is $\bO_n(\bC)$-equivariant. Since $M_n$ maps into the $H_n$ invariants of $T(M)$, we see that $f_n$ maps into $S(V)_n=V^{H_n}$. It is clear that the maps $f_n$ are compatible with transition maps, and so yield a map of $\cC$-modules $M \to S(V)$.

Conversely, suppose that $g \colon M \to S(V)$ is a map of $\cC$-modules. Applying $T$, we obtain a map $T(M) \to T(S(V))$. Since $V$ embeds into a sum of $T^n$'s, every element of $V$ is invariant under $H_n$ for $n$ sufficiently large, and so $T(S(V))=V$. We thus have a map $T(M) \to V$ of $\bO_{\infty}(\bC)$-representations. One easily checks that this construction is inverse to the one in the previous paragraph.

(d) Let $F^i(V)=\{\rR^i \Gamma_n(V)\}_{n \ge 0}$. This is a well-defined $\cC$-module for the same reason $S$ is well-defined. Both $F^{\bullet}$ and $\rR^{\bullet} S$ are cohomological $\delta$-functors. The latter is universal, by definition. However, the former is also universal, since $F^i(V)=0$ for $i>0$ if $V$ is injective. Thus the two are isomorphic.

(e) Since $T$ is exact and kills $\cA_{\tors}$, it induces an functor $\ol{T} \colon \cA/\cA_{\tors} \to \Rep(\bO)$, by the universal property of Serre quotient categories. Let $\ol{S} \colon \Rep(\bO) \to \cA/\cA_{\tors}$ be the composition of $S$ with the localization functor $\cA \to \cA/\cA_{\tors}$. From (c), we have that $\ol{S} \ol{T} = \id$. We now show that the canonical map $\id \to \ol{S} \ol{T}$ is an isomorphism. For this, it is enough to show that for every $M \in \cA$ the kernel and cokernel of the canonical map $M \to S(T(M))$ is torsion. This is clear for the kernel; we now prove it for the cokernel.

Let $\cX$ be the class of objects $M \in \cA$ for which the cokernel of $M \to S(T(M))$ is torsion. Then $\cX$ contains that objects $T^n$ by (a), and therefore all direct sums of such objects since $S$ and $T$ commute with direct sums. It is therefore enough to show that any subquotient of an object in $\cX$ again belongs to $\cX$. Thus suppose we have an exact sequence
\begin{displaymath}
0 \to M_1 \to M_2 \to M_3 \to 0
\end{displaymath}
in $\cA$ with $M_2 \in \cX$. Applying $ST$, we obtain a commutative diagram
\begin{displaymath}
\xymatrix{
0 \ar[r] & M_1 \ar[r] \ar[d]^{f_1} & M_2 \ar[r] \ar[d]^{f_2} & M_3 \ar[r] \ar[d]^{f_3} & 0 \\
0 \ar[r] & ST(M_1) \ar[r] & ST(M_2) \ar[r] & ST(M_3) \ar[r] & \rR^1ST(M_1) }
\end{displaymath}
Let $K_i=\ker(f_i)$ and $C_i=\coker(f_i)$. The snake lemma gives an exact sequence
\begin{displaymath}
0 \to K_1 \to K_2 \to K_3 \to C_1 \to C_2 \to C_3 \to \rR^1ST(M_1)
\end{displaymath}
We have already explained that the $K_i$'s are torsion. Furthermore, $C_2$ is torsion by assumption. Thus $C_1$ is torsion, and so $M_1 \in \cX$. To show that $C_3$ is torsion (and thus obtain $M_3 \in \cX$), it suffices to show that $\rR^1ST(M_1)$ is torsion.

We show, generally, that $\rR^1S(V)$ is torsion for any $V \in \Rep(\bO)$. First suppose that $V$ has finite length. Then $\rR^1S(V)_n=0$ for $n \gg 0$ by Theorem~\ref{thm:Gamma}(d), and so $\rR^1S(V)$ is torsion. Every object of $\Rep(\bO)$ is the union of its finite length subobjects. Since $\rR^1S$ commutes with filtered colimits by Proposition~\ref{prop:colimGamma} and any colimit of torsion modules is torsion, the result follows.
\end{proof}

\begin{corollary} \label{cor:Tinj}
The objects $T^n$ and $\bS_{\lambda}$ are injective in $\cA$.
\end{corollary}

\begin{proof}
Since $T$ is exact, its right adjoint $S$ carries injectives to injectives. Since $T^n_{\infty}$ is injective in $\Rep(\bO)$ (Theorem~\ref{thm:repO}(b)) and $T^n=S(T^n_{\infty})$, we see that $T^n$ is injective. Since $\bS_{\lambda}$ is a summand of $T^n$, it too is injective.
\end{proof}

\begin{corollary}
The category $\cA$ has Krull--Gabriel dimension $1$.
\end{corollary}

\begin{proof}
The subcategory of $\cA$ consisting of objects of Krull--Gabriel dimension~0 is $\cA_{\tors}$, and $\cA/\cA_{\tors}=\Rep(\bO)$ has Krull--Gabriel dimension~0 by Theorem~\ref{thm:repO}(a). Thus $\cA$ has Krull--Gabriel dimension~1 by the recursive characterization of this dimension.
\end{proof}

\begin{remark}
It follows from Theorem~\ref{thm:ST}(e) that $S$ induces an equivalence of categories between $\Rep(\bO)$ and its image in $\cA$, which consists of the {\bf saturated} objects in $\cA$. This statement was also made in \cite[\S 1.2.5]{infrank}. While this image is an abelian category, it is not an abelian subcategory of $\cA$.
\end{remark}

\subsection{Finiteness properties of the section functor}

Our goal in this section is to prove the following finiteness theorem:

\begin{theorem} \label{thm:satfin}
Let $V \in \Rep(\bO)$ have finite length.
\begin{enumerate}[\rm \indent (a)]
\item $S(V)$ is a finitely generated object of $\cA$.
\item $\rR^i S(V)$ is a finitely generated torsion object of $\cA$, for all $i \ge 1$.
\end{enumerate}
\end{theorem}

We begin with a lemma. Recall that $L_{\infty}^{\lambda}$ is the simple object of $\Rep(\bO)$ indexed by $\lambda$. We let $L^{\lambda}$ be the object of $\cA$ given by $S(L^{\lambda}_{\infty})$.

\begin{lemma} \label{lem:Lfg}
Every subobject of $L^{\lambda}$ is finitely generated.
\end{lemma}

\begin{proof}
Let $M \ne 0$ be a subobject of $L^{\lambda}$, let $n$ be minimal so that $M(\bC^n) \ne 0$, and let $v \in M(\bC^n)$ be any non-zero element. Consider a morphism $f \colon \bC^n \to \bC^m$ in $\cC$. Then $f_*(v)$ is non-zero since $L^{\lambda}$ is torsion-free, being in the image of $S$. Since $L^{\lambda}(\bC^m)$ is an irreducible representation of $\bO_m$ by Theorem~\ref{thm:Gamma}(c), it follows that $f_*(v)$ generates every element of $L^{\lambda}(\bC^m)$ as a $\bO_m$-representation, and thus every element of $M(\bC^m)=L^{\lambda}(\bC^m)$. The result follows.
\end{proof}

\begin{proof}[Proof of Theorem~\ref{thm:satfin}]
(a) We proceed by induction on length. Let $V \ne 0$ be given, and consider a short exact sequence
\begin{displaymath}
0 \to W \to V \to L \to 0
\end{displaymath}
with $L$ simple. We obtain an exact sequence
\begin{displaymath}
0 \to S(W) \to S(V) \to S(L).
\end{displaymath}
By Lemma~\ref{lem:Lfg}, the image of the map $S(V) \to S(L)$ is finitely generated. By the inductive hypothesis, $S(W)$ is finitely generated. Thus $S(V)$ is finitely generated.

(b) This follows immediately from Theorem~\ref{thm:Gamma}(d).
\end{proof}

\subsection{Noetherianity}

We now prove a fundamental finiteness result about $\cA$:

\begin{theorem} \label{thm:noeth}
The category $\cA$ is locally noetherian: any subobject of a finitely generated object is again finitely generated.
\end{theorem}

\begin{lemma}
Let $\wt{N} \in \cA$ be finitely generated and let $N$ be a subobject such that $\wt{N}/N$ has finite length. Then $N$ is also finitely generated.
\end{lemma}

\begin{proof}
For $M \in \cA$ let $M_{\ge n}$ be the subobject of $M$ defined by $(M_{\ge n})_k=M_k$ if $k \ge n$ and $(M_{\ge n})_k=0$ for $k<n$. If $M$ is finitely generated then so is $M_{\ge n}$, for any $n$: indeed, the generators of $M$ of degree $>n$ together with a basis for the space $M_n$ generate $M_{\ge n}$. Now, let $n$ be sufficiently large so that $(\wt{N}/N)_k=0$ for $k \ge n$. Then $N_{\ge n}=\wt{N}_{\ge n}$, and is thus finitely generated. Consider the exact sequence
\begin{displaymath}
0 \to N_{\ge n} \to N \to N/N_{\ge n} \to 0.
\end{displaymath}
The left term is finitely generated and the right term has finite length. We thus see that $N$ is finitely generated.
\end{proof}

\begin{proof}[Proof of Theorem~\ref{thm:noeth}]
We first show that $M=T^n$ is noetherian. Thus let $N$ be a subobject. Since $T(N)$ is a subobject of the finite length object $T(M)$, it is finite length, and so $S(T(N))$ is finitely generated by Theorem~\ref{thm:satfin}. The map $N \to S(T(N))$ is injective, since $N$ is torsion-free, and has torsion cokernel by the proof of Theorem~\ref{thm:ST}. Since $S(T(N))$ is finitely generated, the cokernel has finite length. Thus $N$ is finitely generated by the lemma, and so $M$ is noetherian.

Suppose now that $M$ is an arbitrary finitely generated object of $\cA$. By definition, we can express $M$ as a quotient of a subobject $K$ of a direct sum of $\bigoplus_{i \in I} F_i$, where each $F_i$ has the form $T^n$. For a finite subset $J$ of $I$, let $F_J=\bigoplus_{j \in J} F_j$ and $K_J=K \cap F_J$. Since $M$ is finitely generated and $K$ is the union of the $K_J$'s, it follows that there is a $J$ such that $K_J \to M$ is surjective. Since each $F_j$ is noetherian by the previous paragraph, so is the finite sum $F_J$. Since noetherianity descends to subquotients, we see that $M$ is noetherian.
\end{proof}

\begin{remark}
Patzt \cite{patzt} has proved a similar, but more general, result in the general linear and symplectic cases; see \S \ref{s:other}. The symplectic analog likely applies in the present case with minor modifications.
\end{remark}

\subsection{Saturation and local cohomology} \label{ss:loccoh}

We now apply the theory of \cite[\S 4]{symu1}. Let $\Sigma \colon \cA \to \cA$ be the composition $ST$, the {\bf saturation functor}. Let $\Gamma \colon \cA \to \cA$ be the functor that associates to $M$ the torsion submodule $M_{\tors}$. This is a left-exact functor, and we refer to its derived functor $\rR \Gamma$ as {\bf local cohomology}. To apply the theory of \cite[\S 4]{symu1}, we must verify the technical hypothesis (Inj):

\begin{proposition}
Injective objects of $\cA_{\tors}$ remain injective in $\cA$.
\end{proposition}

\begin{proof}
Let $I$ be a finite length injective object of $\cA_{\tors}$. Suppose $M \subset N$ are objects of $\cA$ and we have a map $f \colon M \to I$. Let $k$ be maximal so that $I_k$ is non-zero. Let $M' \subset M$ be the subobject of $M$ defined by $M'_n=M_n$ for $n>k$ and $M'_n=0$ for $n \le k$, and define $N' \subset N$ similarly. Clearly, $f(M')=0$, and so $f$ factors through $M/M'$. The map $M/M' \to N/N'$ is injective, and both objects are torsion, so $f$ extends to a map $g \colon N/N' \to I$. The composition $N \to N/N' \to I$ extends the original map $f$ to $N$. This shows that $I$ is injective in $\cA$.

Now let $I$ be an arbitrary injective object of $\cA_{\tors}$. Since $\cA_{\tors}$ is locally noetherian, we have a decomposition $I =\bigoplus_{\alpha} I_{\alpha}$ where each $I_{\alpha}$ is indecomposable. Since finitely generated objects of $\cA_\tors$ have finite length, $I_{\alpha}$ is the injective envelope of a simple module in $\cA_\tors$, and thus is of finite length itself \cite[(2.1.5)]{infrank}. Thus each $I_{\alpha}$ is injective in $\cA$ by the previous paragraph. Since $\cA$ is locally noetherian (Theorem~\ref{thm:noeth}), any sum of injectives is injective, and so $I$ is injective.
\end{proof}

\begin{proposition}
The indecomposable injectives of $\cA$ are the Schur functors $\bS_{\lambda}$ and the injective envelopes of simple objects $($which are finite length$)$.
\end{proposition}

\begin{proof}
By \cite[Proposition~4.5]{symu1}, the indecomposable injectives of $\cA$ are exactly those of $\cA_{\tors}$ and those of the form $S(I)$ with $I \in \Rep(\bO)$ indecomposable injective. The indecomposable injectives of $\Rep(\bO)$ are the $\bS_{\lambda}(\bC^{\infty})$ (Theorem~\ref{thm:repO}(c)) and $S(\bS_{\lambda}(\bC^{\infty}))=\bS_{\lambda}$ (Theorem~\ref{thm:ST}(a)).
\end{proof}

\begin{proposition} \label{prop:tri}
For $M \in \rD^+(\cA)$ we have a canonical exact triangle
\begin{displaymath}
\rR \Gamma(M) \to M \to \rR \Sigma(M) \to.
\end{displaymath}
Furthermore, if $M \in \rD^b_{\fgen}(\cA)$, then so are $\rR \Gamma(M)$ and $\rR \Sigma(M)$. More precisely, $\rR \Sigma(M)$ is represented by a finite length complex whose terms are finite sums of Schur functors, while $\rR \Gamma(M)$ is represented by a finite length complex whose terms are finite length injective objects.
\end{proposition}

\begin{proof}
The existence of the triangle is \cite[Proposition~4.6]{symu1}. Now suppose $M \in \rD^b_{\fgen}(\cA)$. We have already shown finiteness of $\rR \Sigma(M)$ (Theorem~\ref{thm:satfin}), while finiteness of $\rR \Gamma(M)$ follows from this and the triangle. For the more precise statements, we argue as follows. Since $T(M)$ is a finite length complex whose cohomology groups are finite length objects, Theorem~\ref{thm:repO}(c,d) ensures it is quasi-isomorphic to a complex $N^{\bullet}$ whose terms are sums of representations of the form $\bS_{\lambda}(\bC^{\infty})$. We thus have $\rR \Sigma(M)=S(N^{\bullet})$. The claim now follows from the identity $S(\bS_{\lambda}(\bC^{\infty}))=\bS_{\lambda}$. The claim about $\rR \Gamma(M)$ is immediate.
\end{proof}

\begin{corollary}
A finitely generated object of $\cA$ has finite injective dimension.
\end{corollary}

\subsection{The Grothendieck group} \label{ss:K}

We now discuss the Grothendieck group $\rK(\cA)$ of $\cA$.

\begin{theorem} \label{thm:Ktheory}
We have a canonical isomorphism
\begin{displaymath}
\rK(\cA) = \rK(\Rep(\bO)) \oplus \bigoplus_{n \ge 0} \rK(\Rep(\bO_n)).
\end{displaymath}
\end{theorem}

\begin{proof}
From \cite[Proposition~4.17]{symu1}, we obtain a canonical isomorphism $\rK(\cA)=\rK(\cA/\cA_{\tors}) \oplus \rK(\cA_{\tors})$. We have $\rK(\cA/\cA_{\tors})=\rK(\Rep(\bO))$ by Theorem~\ref{thm:ST}. Every finitely generated object of $\cA_{\tors}$ admits a finite filtration where the graded pieces are supported in a single degree. The category of objects of $\cA$ supported in degree $n$ is equivalent to $\Rep(\bO_n)$, and so the result follows.
\end{proof}

\begin{remark} \label{rmk:Ktheory}
The isomorphism in Theorem~\ref{thm:Ktheory} can be described explicitly as follows. The class of $[V] \in \rK(\Rep(\bO))$ corresponds to $[\rR S(V)] \in \rK(\cA)$. In particular, when $V=\bS_{\lambda}(\bC^{\infty})$ we have $[\rR S(V)]=[\bS_{\lambda}]$, and so the the classes $[\bS_{\lambda}]$ span summand of $\rK(\cA)$ corresponding to $\rK(\Rep(\bO))$. The class of $[V] \in \rK(\Rep(\bO_n))$ corresponds to the class $[M] \in \rK(\cA)$, where $M$ is the $\cC$-module with $M_n=V$ and $M_k=0$ for $k \ne n$.
\end{remark}

\begin{remark}
The Grothendieck group of $\SO_n$ is a polynomial ring in $\lfloor n/2 \rfloor$ variables, generated by exterior powers of the standard representation. This follows from the classification of finite-dimensional representations of a semisimple group by highest weight theory. The Grothendieck group of $\bO_n$ is slightly more complicated, since $\bO_n$ is not connected.
\end{remark}

Given an object $M$ of $\cA$, we define the {\bf Hilbert function} of $M$ by $\HF_M(n)=\dim_\bC M(\bC^n)$. We define the {\bf Hilbert series} of $M$ by
\begin{displaymath}
\rH_M(t) = \sum_{n \ge 0} \dim_\bC M(\bC^n) t^n.
\end{displaymath}
Both constructions are additive in short exact sequences, and thus factor through the Grothendieck group $\rK(\cA)$.

\begin{theorem} \label{thm:hilbert}
Let $M$ be a finitely generated object in $\cA$. Then there exists a polynomial $p \in \bQ[t]$ such that
\begin{displaymath}
p(n)=\dim_\bC M(\bC^n) - \sum_{i \ge 0} (-1)^i \dim \rR^i \Gamma(M)_n
\end{displaymath}
holds for all integers $n \ge 0$. In particular, $\HF_M(n)=p(n)$ for $n \gg 0$, and $\rH_M(t)$ is a rational function whose denominator is a power of $1-t$.
\end{theorem}

\begin{proof}
It suffices to verify the theorem for classes $[M]$ that span $\rK(\cA)$. First suppose that $M$ is a torsion module. Then $\Gamma(M)=M$ (obvious) and $\rR^i \Gamma(M)=0$ for $i>0$ (see \cite[Proposition~4.2]{symu1}). It follows that the identity holds with $p=0$.

Now suppose that $M=\bS_{\lambda}$; note that the classes $[\bS_{\lambda}]$ and the torsion classes span $\rK(\cA)$. Then $\Sigma(M)=M$ (Theorem~\ref{thm:ST}(a)) and $\rR^i \Sigma(M)=0$ for $i>0$ since $M$ is injective (Corollary~\ref{cor:Tinj}), and so $\rR \Gamma(M)=0$ by Proposition~\ref{prop:tri}. We thus take $p(n)=\dim \bS_{\lambda}(\bC^n)$. This is well-known to be a polynomial function of $n$.
\end{proof}

\begin{remark}
In words, the above theorem shows that the Hilbert function coincides with a polynomial (called the Hilbert polynomial) for large values of $n$, and that the discrepancy between the Hilbert function and Hilbert polynomial at small values of $n$ is calculated by local cohomology.
\end{remark}

\subsection{Pointwise algebraic $\cC$-modules} \label{ss:pointwise}

A $\cC$-module $M$ is {\bf pointwise algebraic} if $M_n$ is an algebraic representation of $\bO_n$ for each $n$. Every algebraic representation is pointwise algebraic: this follows immediately from the definitions. We now show that the converse is not true.

Let $M^{(k)}_n=\Sym^{2k}(\bC^n)$. This is an algebraic representation of $\bO_n$. The obvious transition maps $M^{(k)}_n \to M^{(k)}_{n+1}$ give $M^{(k)}$ the structure of a $\cC$-module, and as such it is clearly algebraic. There is a natural surjection $M^{(k+1)}_n \to M^{(k)}_n$ coming from the $\bO_n$-coinvariant in $\Sym^2(\bC^n)$. Let $M_n$ be the inverse limit of the $M^{(k)}_n$'s, taken in the category of algebraic $\bO_n$ representations. Explicitly, $\Sym^{2k}(\bC^n)$ decomposes as a sum of $k+1$ distinct irreducible representations of $\bO_n$, and the inverse limit is simply the direct sum of all of these irreducibles. One easily sees that $M=(M_n)_{n \ge 0}$ is a $\cC$-module, and it is pointwise algebraic by definition. We now show that it is not algebraic.

Each $M^{(k)}$ is generated in degrees $\le 1$. Indeed, considering $M^{(k)}_n$ as the space of degree $2k$ polynomials in variables $x_1, \ldots, x_n$, the image of the transition map $M^{(k)}_1 \to M^{(k)}_n$ is the span of $x_1^{2k}$, which generates $M^{(k)}_n$ as a representation of $\bO_n$. It follows that $M$ is generated in degrees $\le 1$; indeed, the $\bO_n$-subrepresentation of $M_n$ generated by $M_1$ surjects on $M_n^{(k)}$ for all $k$, and is therefore equal to $M_n$ by the explicit description given in the previous paragraph. Since $M_1$ is one-dimensional, we see that $M$ is finitely generated as a $\cC$-module. Since $M_n$ is not finite dimensional for $n>1$, it therefore cannot be algebraic.

Similar examples can be constructed for the other versions of $\cC$ defined in \S\ref{s:other}. In particular, this shows that the expectations outlined in \cite[(3.4.10)]{infrank} are not correct.

\begin{remark}
Let $A$ be the tca $\Sym(\Sym^2(\bC^{\infty}))$, so that $\Rep(\bO)=\Mod_A^0$ (see the proof of Proposition~\ref{prop:colimGamma}). The section functor $S$ can be regarded as a functor from $\Mod_A^0$ to the category of algebraic $\cC$-modules. Given an arbitrary $A$-module $N$, the quotient $N^{\le n}$ (as defined in Proposition~\ref{prop:colimGamma}) is supported at~0, and so $S(N^{\le n})$ is an algebraic $\cC$-module. Define $S(N)$ to be the inverse limit of these modules, taken in the category of pointwise algebraic $\cC$-modules (this category is Grothendieck, and thus complete). This construction extends $S$ to a continuous functor from $\Mod_A$ to the category of pointwise algebraic $\cC$-modules. The module $M$ constructed above is simply $S(A/\fa_1)$, where $\fa_1$ is the first determinantal ideal. It would be interesting to more closely study the connection between $\Mod_A$ and the category of pointwise algebraic $\cC$-modules. For instance, is the former the Serre quotient of the latter by a category of torsion modules, with $S$ being the section functor?
\end{remark}

\section{Other cases} \label{s:other}

\subsection{Symplectic groups}

The theory in \S \ref{s:orth} applies almost without modification to the symplectic case. We define $\cC_{\Sp}$ to be the category of finite dimensional vector spaces equipped with non-degenerate skew-symmetric forms. We let $\cA_{\Sp}$ be the category of algebraic $\cC_{\Sp}$-modules, using an analogous definition of algebraic. The related category of algebraic representations of the infinite symplectic group is studied in \cite[\S 4]{infrank}. The obvious analogs of Theorems~\ref{thm:ST}, \ref{thm:satfin}, \ref{thm:noeth}, and their corollaries hold for $\cA_{\Sp}$, with exactly analogous proofs. The material in \S\S \ref{ss:loccoh}, \ref{ss:K} carries over as well. We mention one slight difference: the description of the Grothendieck group is now
\begin{displaymath}
\rK(\cA_\Sp) = \rK(\Rep(\Sp)) \oplus \bigoplus_{n \ge 0} \rK(\Rep(\Sp_{2n})).
\end{displaymath}
The group $\rK(\Rep(\Sp_{2n}))$ is a polynomial ring in $n$ variables on the classes of exterior powers of the standard representation.

The category $\cC_{\Sp}$ appears in \cite[\S 4.4]{infrank}, as $T_1'$. It also appeared in \cite{putman-sam} and \cite{patzt} under the notation ${\bf SI}(\bC)$. See \cite{patzt} for some applications. We note that \cite[Theorem~D]{patzt} proves the category of pointwise algebraic $\cC_{\Sp}$-modules is locally noetherian, which encompasses our Theorem~\ref{thm:noeth}.

\subsection{General linear groups} \label{ss:gl}

The theory in \S \ref{s:orth} also carries over to the case of algebraic representations of the general linear group, with very minor modifications. Let $\cC_{\GL}$ be the following category. Objects are triples $(V,W,\omega)$ where $V$ and $W$ are finite dimensional and $\omega \colon V \otimes W \to \bC$ is a perfect pairing. A morphism $(V,W,\omega) \to (V',W',\omega')$ is a pair of linear maps $V \to V'$ and $W \to W'$ that respects the pairing. We let $T^{n,m}$ be the $\cC_{\GL}$-module taking $(V,W)$ to $V^{\otimes n} \otimes W^{\otimes m}$. We say that a $\cC_{\GL}$-module is algebraic if it is a subquotient of a direct sum of $T^{n,m}$'s, and let $\cA_{\GL}$ be the category of algebraic representations. We note that the analog of the $\cC_{\bO}$-module $\bS_{\lambda}$ in the present setting is the $\cC_{\GL}$-module given by $(V,W) \mapsto \bS_{\lambda}(V) \otimes \bS_{\mu}(W)$.

The category $\cC_{\GL}$ has an alternate description that is worth pointing out. Let $\cC'_{\GL}$ be the following category. Objects are finite dimensional complex vector spaces. A morphism $V \to W$ is a pair $(i,p)$ where $i \colon V \to W$ and $p \colon W \to V$ are linear maps such that $pi=\id_V$. The functor $\cC'_{\GL} \to \cC_{\GL}$ taking $V$ to $(V,V^*,\omega)$, with $\omega$ the canonical pairing, is an equivalence. The $\cC_{\GL}$-module $T^{n,m}$ corresponds to the $\cC'_{\GL}$-module given by $V \mapsto V^{\otimes n} \otimes (V^*)^{\otimes m}$. Thus $\cC'_{\GL}$-modules are like Schur functors but allow one to use duals. We note that the related theory of algebraic representations of the infinite general linear group is studied in \cite[\S 3]{infrank}.

Theorem~\ref{thm:ST} applies to $\cA_{\GL}$ with obvious changes (e.g., replace $T^n$ with $T^{n,m}$). Theorems~\ref{thm:satfin} and~\ref{thm:noeth} apply without change. The material in \S \ref{ss:loccoh} applies without change as well. Theorem~\ref{thm:Ktheory} applies with the obvious modifications: we obtain
\begin{displaymath}
\rK(\cA_\GL) = \rK(\Rep(\GL)) \oplus \bigoplus_{n \ge 0} \rK(\Rep(\GL_n)).
\end{displaymath}
The Grothendieck group of $\Rep(\GL)$ can naturally be identified with $\Lambda \otimes \Lambda$, where $\Lambda$ is the ring of symmetric functions, see \cite[\S 2]{koike}. Furthermore, the Grothendieck group of $\GL_n$ is isomorphic to $\bZ[z_1,\dots,z_{n-1}, z_n, z_n^{-1}]$ a polynomial ring in $n$ variables with the last variable inverted, where $z_i$ is the class of the exterior power representation $\bigwedge^i \bC^n$. Finally, the material on Hilbert functions holds in the present setting: the Hilbert function is defined to be the function $n \mapsto \dim M(\bC^n, \bC^n)$.

The category $\cC_{\GL}$ appears in \cite[(3.4.6)]{infrank} as $T_1'$. The category $\cC'_{\GL}$ appears there as well, without name. The category $\cC'_{\GL}$ also appears in \cite{putman-sam} and \cite{patzt} as $\VIC$ (or $\VIC(\bC,\bC^{\times})$). We note that \cite[Theorem~C]{patzt} proves that the category of pointwise algebraic $\cC_{\GL}$-modules is locally noetherian, which encompasses our Theorem~\ref{thm:noeth}.

\subsection{Symmetric groups}

There is also a symmetric group analog of the theory in this paper. Let $\cC_{\fS}$ be the category whose objects are tuples $(A,m,\Delta,\eta)$ where $A$ is a finite dimensional vector space, $m \colon A \otimes A \to A$ is an associative commutative multiplication on $A$ (unit not required), $\Delta \colon A \to A \otimes A$ is a coassociative cocommutative coalgebra structure on $A$ with counit $\eta$ such that $m\Delta=\id$ and $\Delta m=(m \otimes 1)(1 \otimes \Delta)$. The motivation for considering this category comes from considering the partition algebra (or really, a categorical form of it) as in \cite[\S 6]{infrank}. We show in \cite[(6.4.7)]{infrank} that $\cC_{\fS}$ is in fact equivalent to the category $\FI$ of finite sets and injections. Let $\cA_{\fS}$ be the category of $\cC_{\fS}$-modules (or equivalently, the category of $\FI$-modules). There is no algebraicity requirement now. We let $T^n$ be the $\cC_{\fS}$-module freely generated in degree $n$. This has an action of the symmetric group $S_n$, and its $\lambda$-isotypic piece is the analog of $\bS_{\lambda}$.

Let $\Rep(\fS)$ be the category of algebraic representations of the infinite symmetric group, as defined in \cite[\S 6]{infrank}. The analogs of the material in \S \ref{ss:repO}, including the theory of specialization, is developed in \cite[\S 6]{infrank} (which relies on some results from \cite{symc1}). The obvious analogs of Theorems~\ref{thm:ST}, \ref{thm:satfin}, \ref{thm:noeth}, and their corollaries hold for $\cA_{\fS}$. The material in \S\S \ref{ss:loccoh}, \ref{ss:K} carries over as well. We note that the result on the Grothendieck group can now be stated as $\rK(\cA_{\fS})=\Lambda \oplus \Lambda$ as $\rK(\Rep(\fS))$ and $\bigoplus_{n \ge 0} \rK(\Rep(\fS_n))$ are both isomorphic to $\Lambda$.

The category of $\FI$-modules was introduced in \cite{fimodule}, which proved some of the above results. It was studied in greater detail (and from a different perspective) in \cite{symc1}, where all of the above results were proved. The follow-up paper \cite{symu1} generalizes many of these results to $\FI_d$-modules.


\begin{thebibliography}{SSW}

\bibitem[CEF]{fimodule}
Thomas Church, Jordan S. Ellenberg, Benson Farb, FI-modules and stability for representations of symmetric groups, {\it Duke Math. J.} {\bf 164} (2015), no.~9, 1833--1910, \arxiv{1204.4533v4}.

\bibitem[DPS]{koszulcategory}
  Elizabeth Dan-Cohen, Ivan Penkov, Vera Serganova, A Koszul category of representations of finitary Lie algebras, {\it Adv. Math.} {\bf 289} (2016), 250--278, \arxiv{1105.3407v2}.

\bibitem[FH]{fultonharris} William Fulton, Joe Harris, {\it Representation Theory: A First Course}, Graduate Texts in Mathematics {\bf 129}, Springer-Verlag, New York, 1991.

\bibitem[Gab]{gabriel} Pierre Gabriel, Des cat\'egories ab\'eliennes, {\it Bull. Soc. Math. France} {\bf 90} (1962), 323--448.

\bibitem[Ko]{koike} Kazuhiko Koike, On the decomposition of tensor products of the representations of the classical groups: by means of the universal characters, {\it Adv. Math.} {\bf 74} (1989), no.~1, 57--86.

\bibitem[KT]{koiketerada} 
Kazuhiko Koike, Itaru Terada, Young-diagrammatic methods for the representation theory of the classical groups of type $B_n$, $C_n$, $D_n$, 
{\it J. Algebra} {\bf 107} (1987), no.~2, 466--511.

\bibitem[NSS]{sym2noeth} Rohit Nagpal, Steven V Sam, Andrew Snowden, Noetherianity of some degree two twisted commutative algebras, {\it Selecta Math. (N.S.)} {\bf 22} (2016), no.~2, 913--937, \arxiv{1501.06925v2}.

\bibitem[Pa]{patzt} Peter Patzt, Representation stability for filtrations of Torelli groups, \arxiv{1608.06507v2}.

\bibitem[PSa]{putman-sam} Andrew Putman, Steven~V Sam, Representation stability and finite linear groups, {\it Duke Math. J.} {\bf 166} (2017), no.~13, 2521--2598, \arxiv{1408.3694v3}.

\bibitem[PSe]{penkovserganova} Ivan Penkov, Vera Serganova, Categories of integrable $sl(\infty)$-, $o(\infty)$-, $sp(\infty)$-modules, {\it Representation Theory and Mathematical Physics}, Contemp. Math. {\bf 557}, AMS 2011, pp. 335--357, \arxiv{1006.2749v1}.

\bibitem[PSt]{penkovstyrkas} Ivan Penkov, Konstantin Styrkas, Tensor representations of classical locally finite Lie algebras, {\it Developments and trends in infinite-dimensional Lie theory}, Progr. Math. {\bf 288}, Birkh\"auser Boston, Inc., Boston, MA, 2011, pp. 127--150, \arxiv{0709.1525v1}.

\bibitem[SS1]{symc1} Steven~V Sam, Andrew Snowden, GL-equivariant modules over polynomial rings in infinitely many variables, {\it Trans. Amer. Math. Soc.} {\bf 368} (2016), 1097--1158, \arxiv{1206.2233v3}.

\bibitem[SS2]{expos}
Steven~V Sam, Andrew Snowden, Introduction to twisted commutative algebras, \arxiv{1209.5122v1}.

\bibitem[SS3]{infrank}
Steven~V Sam, Andrew Snowden, Stability patterns in representation theory, {\it Forum Math. Sigma} {\bf 3} (2015), e11, 108 pp., \arxiv{1302.5859v2}.

\bibitem[SS4]{catgb} Steven~V Sam, Andrew Snowden, Gr\"obner methods for representations of combinatorial categories, {\it J. Amer. Math. Soc.} {\bf 30} (2017), 159--203, \arxiv{1409.1670v3}.

\bibitem[SS5]{symu1} Steven~V Sam, Andrew Snowden, GL-equivariant modules over polynomial rings in infinitely many variables. II, {\it Forum Math. Sigma} {\bf 7} (2019), e5, 71pp., \arxiv{1703.04516v1}.

\bibitem[SSW]{ssw} Steven~V Sam, Andrew Snowden, Jerzy Weyman, Homology of Littlewood complexes, {\it Selecta Math. (N.S.)} {\bf 19} (2013), no.~3, 655--698, \arxiv{1209.3509v2}.

\end{thebibliography}
\end{document}